\documentclass[11pt]{amsart}
\usepackage{amscd,amssymb,amsmath}
\usepackage[all]{xy}
\usepackage{color}
\newtheorem{theorem}{Theorem}[section]
\newtheorem{lemma}[theorem]{Lemma}
\theoremstyle{definition}

\newtheorem{example}[theorem]{Example}
\newtheorem{proposition}[theorem]{Proposition}
\newtheorem{corollary}[theorem]{Corollary}
\newtheorem{conjecture}[theorem]{Conjecture}
\newtheorem{question}[theorem]{Question}

\newtheorem{fact}{Fact}  %level of Exercise
\theoremstyle{remark}
\newtheorem{remark}{Remark}[section]

\newcommand{\be}{\begin{enumerate}}
\newcommand{\ee}{\end{enumerate}}
\newcommand{\bq}{\begin{question}}
\newcommand{\eq}{\end{question}}
\newcommand{\bcj}{\begin{conjecture}}
\newcommand{\ecj}{\end{conjecture}}
\newcommand{\bc}{\begin{corollary}}
\newcommand{\ec}{\end{corollary}}
\newcommand{\bl}{\begin{lemma}}
\newcommand{\el}{\end{lemma}}
\newcommand{\btl}{\begin{technicalLemma}}
\newcommand{\etl}{\end{technicalLemma}}
\newcommand{\bt}{\begin{theorem}}
\newcommand{\et}{\end{theorem}}
\newcommand{\bp}{\begin{proposition}}
\newcommand{\ep}{\end{proposition}}
\newcommand{\bft}{\begin{fact}}
\newcommand{\eft}{\end{fact}}
\newcommand{\brk}{\begin{remark}}
\newcommand{\erk}{\end{remark}}
\newcommand{\bd}{\begin{Dn}}
\newcommand{\ed}{\end{Dn}}

\newcommand{\coker}{{\rm Coker\ }}
\numberwithin{equation}{section}
\newcommand{\bea} {\begin{eqnarray*}}
\newcommand{\beq} {\begin{equation}}
\newcommand{\bey} {\begin{eqnarray}}
\newcommand{\eea} {\end{eqnarray*}}
\newcommand{\eeq} {\end{equation}}
\newcommand{\eey} {\end{eqnarray}}

\begin{document}
\title[Reidemeister spectrum for  metabelian groups ]{Reidemeister spectrum for  metabelian groups of the form 
 ${Q}^n\rtimes \mathbb Z$ and ${\mathbb Z[1/p]}^n\rtimes \mathbb Z$, $p$ prime. }
\author{Alexander Fel'shtyn}
\address{ Instytut Matematyki, Uniwersytet Szczecinski,
ul. Wielkopolska 15, 70-451 Szczecin, Poland
and Boise State University, 1910
University Drive, Boise, Idaho, 83725-155, USA }
\email{felshtyn@diamond.boisestate.edu, felshtyn@mpim-bonn.mpg.de}
\author {Daciberg L. Gon\c{c}alves}
\address{Dept. de Matem\'atica - IME - USP, Caixa Postal 66.281 - CEP 05314-970,
S\~ao Paulo - SP, Brasil}
\email{dlgoncal@ime.usp.br}

\begin{abstract}

 In this note we study the Reidemeister spectrum for metabelian
groups of the form  ${\mathbb Q}^n\rtimes \mathbb Z$ and ${\mathbb Z[1/p]}^n\rtimes \mathbb Z$. Particular attention is given to  the $R_{\infty}$ property of a subfamily of these groups. We also define a Nielsen spectrum of a space and discuss some  examples.

\end{abstract}

\date{\today}
\maketitle

\section{Introduction} 

Let $\phi:G\to G$ be an automorphism of a group $G$. A class of equivalence  defined  by the relation
$x\sim gx\phi(g^{-1})$ for  $x,g \in G$ is called the \emph{Reidemeister class of}   $\phi$
(or the
$\phi$-\emph{conjugacy class} or {the} \emph{twisted conjugacy class of} $\phi$). The
number  of Reidemeister classes, denoted by  $R(\phi)$,  is called
the \emph{Reidemeister number}
of $\phi$.
The interest in twisted conjugacy relations has its origins, in particular,
in the Nielsen-Reidemeister fixed point theory (see \cite{J,FelshB}),
in Selberg theory (see  \cite{Shokra,Arthur}), and  Algebraic Geometry
(see  \cite{Groth}).
A current important  problem of the field concerns obtaining
a twisted analogue of the  Burnside-Frobenius theorem
\cite{FelHill,FelshB,FelTro,FelTroVer,ncrmkwb,polyc,FelTroObzo},
i.e.,  to show the
equality of 
the Reidemeister number of $\phi$ and the number of fixed points of the
induced homeomorphism  of an appropriate dual object.
One step in this process is
to describe the class of groups $G$  such that $R(\phi)=\infty$
for
any automorphism $\phi:G\to G$.

The work of discovering which groups belong to the mentioned class of
groups was begun by Fel'shtyn and Hill in
\cite{FelHill}.
Later, it was shown by various authors
that the  following groups belong to this class:
(1) non-elementary Gromov hyperbolic groups \cite{ll,FelPOMI};
(2) Baumslag-Solitar groups $BS(m,n) = \langle a,b | ba^mb^{-1} = a^n \rangle$
except for $BS(1,1)$ \cite{FelGon};
(3) generalized Baumslag-Solitar groups, that is, finitely generated groups
which act on a tree with all edge and vertex stabilizers being  infinite cyclic
 \cite{LevittBaums},
(4) lamplighter groups $\mathbb Z_n \wr \mathbb Z$ iff $2|n$ or $3|n$  \cite{gowon1};
(5) the solvable generalization $\Gamma$ of $BS(1,n)$ given by the short exact sequence
$1 \rightarrow \mathbb Z\left[\frac{1}{n}\right] \rightarrow \Gamma \rightarrow \mathbb Z^k \rightarrow 1$,
 as well as any group quasi-isometric to $\Gamma$  \cite{TabWong},
groups which are quasi-isometric to $BS(1,n)$ \cite{TabWong2} (while this property is
 not a quasi-isometry invariant); (6) the 
 R.~Thompson group $F$ \cite{bfg}; (7) saturated weakly branch groups including the Grigorchuk group and the Gupta-Sidki group \cite{FelLeoTro};
(8) mapping class groups, symplectic groups and braids groups \cite{dfg};
(9) relatively hyperbolic groups(in particular the  free products of finitely many finitely generated
groups) \cite{f07};
(10)  some classes of finitely generated free nilpotent groups \cite{gw07,ro}
and some classes of finitely generated  free solvable groups \cite{kuro};
(11) some classes of crystallographic groups \cite{DePe}.

The paper \cite{TabWong} suggests a terminology for this property, which we would like to follow.
Namely, ``a group $G$ has \emph{property } $R_\infty$
or is a $R_\infty$ group'', 
if all of its automorphisms  $\phi$
have  $R(\phi)=\infty$.

For an immediate consequences of the $R_\infty$ property for the topological fixed point theory see, e.g., \cite{TabWong2}.

  Following \cite{kuro},   we  define the {\it Reidemeister spectrum of a
 group} $G$, denoted by $Spec(G)$, as the set of 
  natural numbers  $k$ such  that there is an automorphism $\phi\in Aut(G)$ with   $ R(\phi)=k$  ($k$ can be infinite).
  In terms of the spectrum, the $R_{\infty}$-property of the group $G$ simply  means  that 
$Spec(G)$ contains only  one element which is the infinity.

 It is easy to see that  $Spec({\mathbb Z}) = \{2\} \cup \{ \infty \}$, 
and,  for $n \geq 2$,  the spectrum is full, i.e.
$ Spec({\mathbb Z}^n) = {\mathbb N} \cup \{\infty \}$.
Let $N = N_{rl}$ be the free nilpotent group of rank $r$ and class
$l$. Then for $N_{22}$ (also known as  discrete Heisenberg group)  $Spec(N_{22}) = 2{\mathbb N} \cup \{\infty \}$
\cite{ind,fit, kuro}.   It is also  known that  $ Spec(N_{23}) = \{ 2k^2 | k \in {\mathbb N}\} \cup
\{\infty \}$ \cite{kuro} and 
$ Spec (N_{32}) = \{2n - 1| n \in {\mathbf N}\} \cup \{
4n | n \in {\mathbf N}\} \cup \{ \infty \}$ \cite{kuro}.

 Let $X=L(m,q_1,\ldots,q_r)$ be a generalized lens space,
and let $f:X\to X$ be  a continuous map of degree 
$d$,  where $\mid d \mid\not= 1$. Let  $f_*: \pi_1(X)\to \pi_1( X)$ be induced homomorphism on the fundamental group 
$\pi_1(X)=\mathbb Z/m\mathbb Z$.  
In   1943   Franz \cite{Franz} has observed that
 $N(f)=R(f)= R(f_*) = \#\coker(1-f_*)=\#(\mathbb Z/m\mathbb Z)/(1-d)(\mathbb Z/m\mathbb Z)= (|1-d|,m)$,
 where $N(f)$ and $R(f)$ are Nielsen and Reidemeister numbers of the map $f$
and where $(|1-d|,m)$ denotes  the gcd of $|1-d|$ and $m$. This gives a strong arithmetic  restriction on the Reidemeister spectrum for endomorphisms of the group $\mathbb Z/m\mathbb Z$.
We observe that  the knowledge of the Reidemeister spectrum of a group  can be quite useful for fixed point theory. 

 In this paper 
 we study the Reidemeister spectrum and the  $R_{\infty}$-property
 for a subfamily of the family of the metabelian groups of the form
${\mathbb Q}^n\rtimes  \mathbb Z$ and ${\mathbb Z[1/p]}^n\rtimes \mathbb Z$ for $p$ a prime. 

We also define a Nielsen spectrum of a space and discuss some  examples.
\medskip\noindent
{\bf Acknowledgement.}
We would like to thank the Max-Planck-Institut f\"ur Mathematik (MPIM)
in Bonn for its kind support and
hospitality while most  of this work was done.

\section{ Preliminaries}

In this section we show that for certain short exact sequences of groups   the kernel is characteristic. Then   we compute  
$Aut(\mathbb Q)$, $Aut(\mathbb Z[1/p])$,  and the Reidemeister spectrum of the groups  $\mathbb Q$ and  $\mathbb Z[1/p]$  where 
$\mathbb Q$ denote the rational numbers and $p$ a prime.

Let us consider a short exact sequence of the form
$ 1 \to \mathbb K \to G \to \mathbb Z \to 1$, which of course splits.

\begin{lemma} Suppose  the group $\mathbb K$ has the property that for every $x\in \mathbb K$ there is a natural number $s>1$ such that 
$x$ is divisable by $s$ $($i.e. there is $y\in \mathbb K$ such that $y^s= x$ $)$. Then $\mathbb K$ is characteristic in $G$.
\end{lemma}
\begin{proof} Let $\phi:G \to G$ be an automorphism. We know that $G\approx K\rtimes_{\theta}\mathbb Z$ for some homomorphism $\theta:\mathbb Z \to Aut(\mathbb K)$.
 Let $x\in \mathbb K$ and $\phi(x)=(z,r)$. We will
 show that $r=0$. It follows from the definition of the operation on the semi-direct product  that,  if $s$ divides $x$,  then $s$ also divides   
$\phi(x)=(z,r)$.
Again,  by the definition of the operation on the semi-direct product,   $s$  divides $r$.   From the hypothesis it  follows that there is 
an infinite sequence of positive integers  
such that  for  each integer $s$ of the sequence, $r$ is divisible by  this integer. Therefore, $r$ has an infinite number of divisors and must 
be zero.  
\end{proof}

For the rational numbers regarded as an abelian group with the additive operation we have:

\begin{lemma} The group of automorphims of  $\mathbb Q$ is isomorphic to the multiplicative group of the  rationals  different from zero,  denoted by $\mathbb Q^*$.
 \end{lemma}
\begin{proof} First,   we can  observe that an automorphism $\phi: \mathbb Q \to \mathbb Q$ is determined if we know  its  value  at 1. 
To see  this  let us consider an arbitrary element $p/q  \in \mathbb Q$. Then $\phi(p/q)=p\phi(1/q)$ and  it suffices to determine     
$\phi(1/q)$. Since $\mathbb Q$ is torsion free, it  follows that the divisibility in $\mathbb Q$ is unique and hence  $\phi(1/q)$ is uniquely 
determined. Conversely,  a multiplication by any rational number different from zero provides an automorphism, since the  multiplication by the inverse number  provides the inverse 
homomorphism. 
 \end{proof}

Now  let    $\mathbb Z[1/p]$ for $p$ a prime be the subring of the rationals and also by the same notation we denote the abelian additive group. 
By  $Aut(\mathbb Z[1/p])$  we mean   the automorphism group  of $\mathbb Z[1/p]$  as abelian group.

\begin{lemma} The group  $Aut(\mathbb Z[1/p])$ is isomorphic to the multiplicative set  of the elements  of $\mathbb Z[1/p]$  generated by $\{\pm p\}$, which is isomorphic 
to the group $\mathbb Z+\mathbb Z/2\mathbb Z$.
 \end{lemma}

\begin{proof} The first part is similar to the proof of the  previous Lemma. Namely, a homomorphism $\phi$ is determined by the value of the homomorphism  at 1, and  it is multiplication by this value.  In order to have $\phi$ be an automorphism  we need that $\phi(1)$ is invertible. Let $r/s\in \mathbb Z[1/p]$ where $r/s$ is written in the reduced form. If $r/s\in \mathbb Z[1/p]$,  then we have either $r=1$ and $s=p^t$ or   $s=1$ and $r=p^t$,  and the result follows.
 \end{proof}

Now we  determine the Reidemeister  spectrum of $\mathbb Q$ and $\mathbb Z[1/p]$. We need a Lemma which will also  be used  for the computation
of the Reidemeister spectrum of other groups.
Given  any $x\in \mathbb Z[1/p]$, $x$  can be written uniquely in the form $\pm q/p^n$ where $n\in \mathbb Z$, $q\in \mathbb N$(the natural numbers) and  $p,q$  relatively prime.
Denote this number $q$ by $v_{P}(x)$. Here $P$ is the set of all primes relatively prime with $p$. Also let us observe that the abelian group 
$\mathbb Z[1/p]^n$ is  a  free $\mathbb Z[1/p]-$module. Then  a group homomorphism of the abelian 
group  $\mathbb Z[1/p]^n$ is always a homomorphism of $\mathbb Z[1/p]-$module.

\begin{lemma} Let $\psi: A^n \to A^n$ be a homomorphism where $A$ is either $\mathbb Q$ or $\mathbb Z[1/p]$. Then 
$\psi$ is invertible if, and only if,  the determinant of the matrix of $\psi$ is
invertible. Furthermore, for $A=\mathbb Q$  the cardinality of the cokernel of $\psi$ is infinite if 
$\det(\phi)=0$,  and it is  1 if $\det(\psi)\ne 0$.  For $A=\mathbb Z[1/p]$ the cardinality 
of the cokernel of $\psi$  is the natural number   $v_{P}(x)$, defined above, for  $x=\det(\psi)$, if $x\ne 0$.
If $\det(\psi)=0$ then the cardinality of the cokernel is infinite.
\end{lemma}  
\begin{proof} That $\psi$ is invertible if and only if $det(\psi)$ is invertible is a classical fact  since $A$ is a 
commutative ring. Therefore if $A=\mathbb Q$ then $det(\psi)\ne 0$ implies $\psi$ invertible and we have the cardinality of the cokernel
1. If $det(\psi)=0$ then the result follows promptly from the fact that  $\mathbb Q$ is an infinite field.

  Let $A=\mathbb Z[1/p]$. 
The matrix $M$ of $\psi$ has entries in $A$. Multiplication by $p^l$ on $A^n$, denoted by $R_{p^l}:A^n\to A^n$, is an automorphism 
for any $l$. So the cokernel of $\psi$ and the cokernel of the  composite of $\psi$ with  the multiplication 
by  $p^l$ are isomorphic.  For a sufficiently  large $l$ we can assume that the matrix of the composite has entries of
the form $xp^i$ for $x$ an  integer(possibly negative) relatively prime with $p$  and $i\geq 0$.  Observe that the determinant of 
the new matrix
$M_1$ is $p^ldet(\psi)$ and  $ v_P$ of the two determinants are the same. Now consider  the homomorphism 
$\psi_1:\mathbb Z^n \to \mathbb  Z^n$ defined by $M_1$ and $\psi_2=R_{p^l}\circ \psi: A^n \to A^n$ the homomorphism defined by the same matrix.

Then we have the following commutative diagram:

\begin{equation*}
\xymatrix{ 
 \ar[r]  & \mathbb Z^n \ar[r]^-{\psi_1} \ar[d]_{\iota} & \mathbb Z^n 
\ar[r]^-{\pi_1} \ar[d]_{\iota} &  coker(\psi_1)  \ar[r] \ar[d] & 0\\
 \ar[r] & \mathbb Z[1/p]^n \ar[r]^-{\psi_2} & \mathbb Z[1/p]^n \ar[r]^-{\pi_2} & coker(\psi_2)   \ar[r] & 0}
\end{equation*}

\noindent where  $coker(\psi_2)$ has no $p$ torsion, as result of the $p-$divisibility of the group 
$\mathbb Z[1/p]$. After we take the tensor product with $\mathbb Z[1/p]$ as  $\mathbb Z-$module the two first vertical 
homomoprhisms becomes  isomorphisms  and the  horizontal lines are  exact. So we 
have an isomorphism between the tensor product of the cokernels. The  cokernel $coker(\psi_1)$  is a finitely generated  abelian group. 
Now we look at the two possible cases.  If  $coker(\psi_1)$ is infinite
then it contains  a copy of $\mathbb Z$  and it  follows that  $coker(\psi_2)$ is also infinite.  If   $coker(\psi_1)$ is finite
then it is  the direct product of two  finite groups where the first has order a power of $p$ and the other has order 
relatively prime to $p$.  After taking  the tensor product we obtain only the finite subgroup of order relatively prime 
to $p$ which is simultaneously the order of the cokernel  of $\psi_2$ and $v_P(det(\psi_1))$. But  
$v_P(det(\psi_1))=v_P(det(\phi_2))=v_P(det(\psi))$ and the result follows.

 The case when $A=\mathbb Q$ is simpler and we leave to the reader.  
\end{proof}

\begin{proposition} a) The Reidemeister spectrum of $\mathbb Q$ is\\ $Spec(\mathbb Q)=\{1\}\cup\{\infty\}$. \\
b) The Reidemeister spectrum of $\mathbb Z[1/p]$ for $p$ an odd  prime is 
$Spec(\mathbb Z[1/p])=\{p^l+1, \  p^{l+1}-1 | l\geq 0 \}  \cup  \{\infty\}$ \\
  and  $Spec(\mathbb Z[1/2])=\{2^l+1, \  2^{l}-1 | l\geq 1\} \cup \{\infty\}$ \\
\end{proposition}
\begin{proof} Part a)- Given an automorphism $\phi: \mathbb Q \to \mathbb Q$ we know that it is multiplication by a rational number $r$. 
If $r=1$ then the cokernel of multiplication of  $r-1=0$ is $\mathbb Q$ so we obtain Reidemeister number infinite. Otherwise
we have multiplication by $r-1\ne 0$ which is surjective. Therefore the Reidemeister number is 1 and the result follows.

Let us consider $Z[1/2]$ and $P=\mathbb P- \{2\}$. An automorphism is multiplication by a rational number $r$ such that $v_P(r)$ is 1. So 
we have $r=\pm 2^l$, $l\in Z$. If $r=1$ then we obtain that the Reidemeiter number is infinite, and for $r=-1$ we obtain 
multiplication by -2 so we obtain Reidemeister number 1. So let $l\ne 0$. Then the numbers $2^l\pm 1$ are always odd and the 
result follows.
  The case $p$ an odd prime is similar and simpler. We leave it to the reader.
 \end{proof}

\section{The semi-direct product $\mathbb Q^n\rtimes \mathbb Z$, $\mathbb Z[1/p]^n\rtimes \mathbb Z$,  $n\leq 2$}

 We begin by recall some basic facts. Given any automorphism  $\phi$ of one of the groups $\mathbb Q^n\rtimes \mathbb Z$, 
$\mathbb Z[1/p]^n\rtimes \mathbb Z$,   we know from section 2 that the subgroup either $\mathbb Q^n$ or $\mathbb Z[1/p]^n$ is characteristic. 
So we obtain a homomorphism of short exact sequence. 
Whenever the induced homomorphism $\bar \phi$ on the quotient is the identity then by well known facts, see 
\cite{go:nil1} this  implies that the  Reidemeister number is infinite. Also from \cite{go:nil1} in the case where $\bar \phi:\mathbb Z \to \mathbb Z$ is multiplication by
 -1(the only other possibility) then the Reidemeister number is computed as the sum of the Reidemeister number of $\phi'$ 
and the Reidemeister number of  $\theta(1)\circ \phi'$. We will use the above procedure for  the calculation which follows.

\subsection{The case $n=1$}

In this subsection we have an action $\theta: \mathbb Z \to A$ where $A$ is either   $\mathbb Q$ or $\mathbb Z[1/p]$.
The homomorphism $\theta$ is completely determined by $\theta(1)$ which, in turn  is determined by its values at $1\in A$.
So we identify $\theta(1)$ with its value at $1\in A$. 

\begin{proposition} a) The    Reidemeister spectrum of $\mathbb Q\rtimes_{\theta}\mathbb Z$ is \break
$Spec(\mathbb Q\rtimes_{\theta}\mathbb Z)=\{\infty\}$ if
$\theta(1)$ is a non zero rational number different from $\pm 1$ and it is  $Spec(\mathbb Q\rtimes\mathbb Z)=\{2\}\cup\{\infty\}$ otherwise. \\
b)  For $p$ an odd prime   the     Reidemeister spectrum of $\mathbb Z[1/p]\rtimes_{\theta}\mathbb Z$ is \\
$Spec(\mathbb Z[1/p]\rtimes_{\theta}\mathbb Z)=\{\infty\}$ if
$\theta(1)\in \mathbb Z[1/p]$ is a non zero invertible element  different from $\pm 1$. If $\theta(1)=1$ then  
$Spec(\mathbb Z[1/p]\rtimes_{\theta}\mathbb Z)    =\{p^l+1, \  p^{l+1}-1 | l\geq 0 \}  \cup  \{\infty\}$.
 If  $\theta(1)=-1$ then
$Spec(\mathbb Z[1/p]\rtimes_{\theta}\mathbb Z)=\{  2p^{l+1}| l\geq 0 \}  \cup  \{\infty\}$.

c) The   Reidemeister spectrum of   $\mathbb Z[1/2]\rtimes_{\theta}\mathbb Z$ is 
$Spec(\mathbb Z[1/2]\rtimes_{\theta}\mathbb Z)=\{\infty\}$ if
$\theta(1)\in \mathbb Z[1/2]$ is a non zero invertible element  different from $\pm 1$. If $\theta(1)=1$ then
 $Spec(\mathbb Z[1/2]\rtimes_{\theta}\mathbb Z)=\{2(2^l+1), \  2(2^{l}-1) | l\geq 1\} \cup \{\infty\}$. \\ 
 If   $\theta(1)=-1$ then  $Spec(\mathbb Z[1/2]\rtimes_{\theta}\mathbb Z))=\{2^{l+1}| l\geq 1\} \cup \{\infty\}$.
\end{proposition}

\begin{proof} Let $\phi:Q\rtimes_{\theta}\mathbb Z$ be a automorphism. From the discussion in the beginning of the section we know that 
$\infty$ belongs to the spectrum and we will look at automorpisms such that $\bar \phi$ is multiplication by -1. We compute the
automorphisms $\phi':\mathbb Q \to \mathbb Q$ which arises as restriction of such automorphisms. In order to an automorphsm $\phi'$ be the
 restriction of an automorphism of the large  group we must have  the relation
$\phi'\circ \theta(1)=\theta(-1)\circ \phi'$. This implies that $kr=r^{-1}k$ where $\phi'(1)=k$(so different from 0) and $\theta(1)=r$.
This implies that $r^2=1$.  So if $r\ne \pm 1$ then there is no such automorphism and the Reidemeister spectrum of the group is 
$\{\infty\}$. If $r=\pm 1$ then  $k$ can assume any non zero value. If $r=1$ then the two automorphisms on the fibers are multiplication 
by $k$ and each one has Reidemeister number 1 if  $k\ne 1$. If $r=-1$  then the two automorphisms on the fibers are multiplication 
by $k$, $-k$ respectively. So for $k\ne \pm 1$ the  Reidemeister number is 2 and part a) follows.

 Part b)  The infinity certainly belongs to the spectrum because  the Reidemeister number of the identity is infinite. The
element    $\theta(1)$ is invertible since it is an isomorphism. As in case a) we have that if $\theta(1)\ne \pm 1$ there is 
no such automorphism and follows that the group has spectrum $\{\infty\}$. Again as in part a) for $\theta(1)=\pm 1$ we 
have $\phi'$ can be an  arbitrary  automorphism of $\mathbb Z[1/p]$.  For $r=1$ the two automorpisms are the same and we obtain as 
Reidmeister number $2(p^l\pm 1)$, $l>0$ and 4. If $r=-1$ then we have to look at $k-1$ and $k+1$ for $k$ an invertible,
 so of the form $\epsilon p^t$. If $t=0$ then we get Reidemeister infinite. If $t\ne 0$ then we get as Reidemeister number 
$2p^l$. So the result follows.

Part c). The proof  is  similar to the proof of case  b) where the only difference is because $2^0+1$ is not a Reidemeister number
of a homomorphism for case c) but $p^0+1$ it is for case b). This justifies why the formulas are slight different. 
\end{proof}

\subsection{The case $n=2$}

First we analyze the case where the kernel is $\mathbb Q^2$. Let $G=\mathbb Q^2\rtimes \mathbb Z$ and 
for  an automorphism $\phi:G \to G$ let $\phi'$ be
the restriction of $\phi$ to $\mathbb Q^2$. If $M$ is the matrix of $\phi'$ then we call $det(\phi')$ the determinant of $M$.

\begin{proposition}  The    Reidemeister spectrum of $\mathbb Q^2\rtimes_{\theta}\mathbb Z$ is either $\{\infty\}$ or $\{2\}\cup\{\infty\}$.
The $Spec(\mathbb Q^2\rtimes_{\theta}\mathbb Z)=\{\infty\}$ if  there is no an invertible matrix $N$ over $\mathbb Q$
such that $N=MNM$, $det(Id-N)\ne 0$ and $det(Id-MN)\ne 0$, where $\theta(1)=M$. Otherwise the 
Reidemeister spectrum is    $\{2\}\cup\{\infty\}$.
\end{proposition}

\begin{proof}The result follows  from the considerations on the beginning  of this section and basic facts  about Reidemeister numbers  for homomorphisms
of $\mathbb Q^2$. 
\end{proof}

   A more explict description of the matrices $N$ in terms of the matrix $M$ which satisfies  the conditions of the Proposition
above is still in progress.  But  we give an example of  a family of groups which have Reidemeister spectrum $\{\infty\}$ 
and also  an example of  another  family of groups which have  Reidemeister spectrum  $\{2\}\cup\{\infty\}$.\\

{\bf Example 1}- Let $\theta(1)=M$ be the automorphism given by 
$$
\begin{pmatrix}
r & 0 \\
0 & s
\end{pmatrix}
$$
\noindent  with $rs\ne 0, 1$ and either  $r^2\ne 1$ or $s^2\ne  1$. There is no automorphism such that the induced homomorphism 
on the quotient is multiplication by -1. This can be proven by showing that the only solutions for the
matrices $N$ which satisfy   $N=MNM$ has the property that  $det(N)=0$.

{\bf Example 2}- Let $\theta(1)$ be the automorphism given by 
$$
\begin{pmatrix}
r & 0 \\
0 & s
\end{pmatrix}
$$
\noindent $r^2\ne 1$, $rs=1$. By direct calculation we can find all matrices $N$ and they are 
of the form  
$$
\begin{pmatrix}
0 & b \\
c & 0
\end{pmatrix}
$$
\noindent for arbitrary $b,c$. The matrix $MN$ is 

$$
\begin{pmatrix}
0 & rb \\
cs & 0
\end{pmatrix}
$$

Then $det(Id-N)=det(Id-MN)=1-bc$ and whenever  $1-bc\ne 0$ we obtain an example
where the  Reidemeister number is 2. Therefore this provides a family of examples of  groups where the 
Reidemeister spectrum  is $\{2\}\cup\{\infty\}$.

{\bf Example 3}- Let $\theta(1)$ be the automorphism given by 
$$
\begin{pmatrix}
0 & u \\
v & 0
\end{pmatrix}.
$$
\noindent  Let us consider two cases. Suppose that 
 $|uv| \ne 1$. Then follows that there is no invertible $N$  such that $N=MNM$ 
and follows that the Reidemeister spectrum of such groups is $\{\infty\}$.

 For the second case let $u=v=1$. Then  
by direct calculation we can find all matrices $N$   
which turns   out to be of the form
$$
\begin{pmatrix}
a & b \\
b & a
\end{pmatrix}
$$
\noindent and the  matrix $MN$ is of the form
$$
\begin{pmatrix}
b & a \\
a & b
\end{pmatrix}.
$$

\noindent The first matrix  has determinant $a^2-b^2$. The determinant of $Id-N$ and $Id-MN$ are respectively
$1-2a+a^2-b^2$ and $1-2b+b^2-a^2$. There are plenty of rational  values of $a$ and $b$ such that  these 3 values are 
different from zero and consequently we obtain for each such values one example of a group which admits 
an automorphism which has Reidemeister number $2$.

 Now we consider the case  
$\mathbb Z[1/p]^2\rtimes_{\theta}\mathbb Z$ for $p$ 
a prime. We will compute   the    Reidemeister spectrum of $\mathbb Z[1/p]^2\rtimes_{\theta}\mathbb Z$ for two 
families of actions $\theta$. 

The  Proposition  3.2 holds partially in this case.

\begin{proposition}  The    Reidemeister spectrum of   $\mathbb Z[1/p]^2\rtimes_{\theta}\mathbb Z$ is $\{\infty\}$ 
if and only if there is no an invertible matrix $N$ over $\mathbb Z[1/p]$
such that $N=MNM$ and both $det(Id-N)$ and $det(Id-MN)$ are non zero. Here $\theta(1)=M$. 
\end{proposition}

\begin{proof} The proof is similar  to the proof of Proposition 3.2.
\end{proof}

 The Reidemeister spectrum for the groups where there exist a solution $N$ is not very simple in general
but we can answer for two large families of actions $\theta$.  For example if $\theta(1)$ is the identity
then the Reidemeister spectrum of this group is obtained  multiplying by 2 the numbers(including infinite) which belong to the spectrum of the 
first factor.
So the first step is to compute the 
Reidemeister spectrum  of  $\mathbb Z[1/p]^2$.

\begin{proposition} The Reidemeister spectrum of  $Z[1/p]^2$ is \\
$Spec(\mathbb Z[1/p]^2)=\{n |  n\in \mathbb N \ and \  (n,p)=1 \}\cup \{\infty\}$  where $(n,p)$ denote 
the gcd of $n$ and $p$.
\end{proposition}

\begin{proof}  Any matrix $N \in Gl(2,\mathbb Z)$ can be regarded as an automorphism
of $\mathbb Z[1/p]^2$. It is well known that the Reidemeister spectrum of $\mathbb Z+\mathbb Z$ is 
$\mathbb N$. So the Reidemeister spectrum of $\mathbb Z[1/p]^2$ contains $v_p(n)$ for every natural number
by Lemma 2.4. So it  contains all positive numbers relatively prime to $p$. But again by Lemma 2.4
an element of the spectrum has to be a positive  integer relatively prime to $p$.
So the result follows. 
\end{proof}

An immediate consequence of the Proposition above is that 
$Spec(\mathbb Z[1/p]^2\times {\mathbb Z)=\{2n | (n,p)=1  \  n\in \mathbb N \}\cup \{\infty}\}$ where $(n,p)$ denote the gcd of $n$ and $p$.

Let us start with  the group $Spec(\mathbb Z[1/2]^2\rtimes_{\theta}\mathbb Z)$.  

\begin{proposition} Let  $\theta(1)$ be of the form:
$$
\begin{pmatrix}
r& 0 \\
0 & s
\end{pmatrix}
$$ Then we have the following cases:

a) If  $r=s=\pm 1$ then  
$Spec(\mathbb Z[1/2]^2\rtimes_{\theta}\mathbb Z)=\{2n | (n,2)=1  \  n\in \mathbb N \}$ where $(n,2)$   denote the gcd of $n$ and $2$.\\
b) If $r=-s=\pm 1$ then  the   Reidemeister spectrum of the group $\mathbb Z[1/2]^2\rtimes_{\theta}\mathbb Z$ is 
$Spec(\mathbb Z[1/2]^2\rtimes_{\theta}\mathbb Z)=\{2^{l+1}, 2^{k}(2^l\pm 1)|  \ l\geq 1,  k\geq 2\}\cup \{\infty\}$.\\
c) If  $rs=1$ and $|r|\ne 1$      
then   $Spec(\mathbb Z[1/2]^2\rtimes_{\theta}\mathbb Z)=\{2(2^k+1), 2(2^l-1) | \ k\geq 0,  \ l\geq 1\}\cup \{\infty\}$\\
d) If either $r$ or $s$        does not have module equal to one, and $ rs \ne 1$ then 
 there is no automorphism of the group 
such that the induced on the quotient is multiplication by $-1$ and follows that  $Spec(\mathbb Z[1/2]^2\rtimes_{\theta}\mathbb Z)=\{\infty\}$.
\end{proposition}
\begin{proof} Let $N$ be the matrix 
$$
\begin{pmatrix}
a& b\\
c& d
\end{pmatrix}.
$$
The equation $N=MNM$ corresponds   to the system $a=ar^2$,   $b=brs$, $c=crs$ and $d=ds^2$. 
Suppose that $|r|\ne 1$ and   $ rs \ne 1$. This implies that $a=b=c=0$ and so the system has no solution for
 a matrix $N$ invertible. Similarly if we assume $|s| \ne 1$. So part d) follows.

For the part a) we have that all invertible matrices $N$ are solutions. So the result follows 
from the Proposition 3.4 above.

For the part b) we have that the matrix $N$ is diagonal. By straightforward calculation  we can assume that the elements 
of the diagonal of $N$ are of the form $ 2^i$ for $i\geq 0$. A direct calculation shows the values for the spectrum.

For the part c) we have that the matrix $M$ is of the form 

$$
\begin{pmatrix}
\epsilon 2^{\ell}& 0\\
0& \epsilon 2^{-\ell}
\end{pmatrix}.
$$

\noindent  and  $N$ of the form 

$$
\begin{pmatrix}
0& \delta_1 2^u\\
\delta_2 2^v & 0
\end{pmatrix}.
$$

\noindent The product $MN$ is given by 

$$
\begin{pmatrix}
0& \epsilon\delta_1 2^{u+\ell}\\
\epsilon\delta_2 2^{v-\ell} & 0
\end{pmatrix}.
$$
 \noindent Follows the  matrices of $Id-N$ and $Id-MN$:

$$
\begin{pmatrix}
1& -\delta_1 2^{u}\\
-\delta_2 2^{v} & 1
\end{pmatrix}.
$$

\noindent and 
$$
\begin{pmatrix}
 1 & -\epsilon\delta_1 2^{u+\ell}\\
-\epsilon\delta_2 2^{v-\ell} & 1
\end{pmatrix}
$$

\noindent respectively. The result follows by straightforward calculation.

\end{proof}

\begin{proposition} Let  $\theta(1)$ be of the form:
$$
\begin{pmatrix}
0 & u \\
v & 0
\end{pmatrix}
$$ Then we have the following cases: 

a) If $ uv=1$     then    the  Reidemeister spectrum of the group $\mathbb Z[1/2]^2\rtimes_{\theta}\mathbb Z$ is 
$Spec(\mathbb Z[1/2]^2\rtimes_{\theta}\mathbb Z)=\{ 2^k(2^l\pm 1)| \ k\geq 2,  l\geq 1\} \cup\{ \infty\}$.\\

b) If $ uv=-1$  then    the  Reidemeister spectrum of the group $\mathbb Z[1/2]^2\rtimes_{\theta}\mathbb Z$ is 
$Spec(\mathbb Z[1/2]^2\rtimes_{\theta}\mathbb Z)=\{2(2^{2m}-1)| \  m>0\} \cup\{ \infty\}$.\\

c) If   $u^2v^2\ne  1$ then    there is no automorphism of the group such that the induced on the quotient is multiplication by $-1$ and follows that  $Spec(\mathbb Z[1/2]^2\rtimes_{\theta}\mathbb Z)=\{\infty\}$.
 \end{proposition}
\begin{proof} In this case the entries of the matrix $N$ 
$$
\begin{pmatrix}
a& b\\
c& d
\end{pmatrix}
$$

\noindent must satisfy the equations $uvd=a, u^2c=b, bv^2=c, auv=d$. It follows that $a=(uv)^2a$,  $d=(uv)^2d$, $b=(uv)^2b$, $c=(uv)^2c$.
So part c) follows promptly from these equations.  

 Let us consider the case a).  In this case the system of equations provide
$a=d$. So the determinant of $N$ becomes $a^2-bc=a^2-b^2v^2=(a+bv)(a-bv)$ which is an invertible element, so $(a+bv)$ and $(a-bv)$ are also 
invertible. Since $uv=1$ follows that  $u=\delta 2^{-t}$ and  $v=\delta 2^{t}$ for some integer $t$ and $\delta \in \{1,-1\}$.   

 The matrix $N$ is of the form 

$$
\begin{pmatrix}
a& b\\
bv^2& a
\end{pmatrix}
$$
 \noindent and the matrix $MN$ is

$$
\begin{pmatrix}
bv& a/v\\
av& bv
\end{pmatrix}.
$$

\noindent Our task is to compute $v_{P}$ of $det(Id-N)$ and $det(Id-MN)$. We have   $det(Id-N)=1-2a+a^2-b^2v^2$ and 
 $det(Id-MN)=1-2bv+b^2v^2-a^2$ so  $-det(Id-MN)=-1+2bv+a^2-b^2v^2$. We compute $v_{P}$ of  $-det(Id-MN)$. 

 Since $a+bv$ and $a-bv$ are  invertible they are of the form $a+bv=\epsilon_12^{l_1}$ and   $a-bv=\epsilon_22^{l_2}$.
Therefore $a=\epsilon_1 2^{l_1-1} +\epsilon_2 2^{l_2-1}$ and $b=\delta(\epsilon_12^{l_1-t-1}-\epsilon_22^{l_2-t-1}$).
Also $det(N)= a^2-b^2v^2=\epsilon_1\epsilon_22^{l_1+l_2}$.

 So we obtain $det(Id-N)= \epsilon_1\epsilon_2 2^{l_1+l_2}-\epsilon_1 2^{l_1}-\epsilon_2 2^{l_2}+1$\\
and  $-det(Id-MN)= \epsilon_1\epsilon_2 2^{l_1+l_2}+\epsilon_12^{l_1}-\epsilon_2 2^{l_2}-1$. It is not difficult to see that the
rational numbers, up to multiplication by a power of 2(positive or negative) are  the integers
%$ \epsilon_1\epsilon_2 2^{|l_1|+|l_2|}-\epsilon_12^{|l_1|}-\epsilon_2 2^{|l_2|}+1$\\
%and  $\epsilon_1\epsilon_2 2^{|l_1|+|l_2|}+\epsilon_12^{|l_1|}-\epsilon_2 2^{|l_2|}-1$, respectively. 
$  2^{|l_1|+|l_2|}-\epsilon_22^{|l_1|}-\epsilon_1 2^{|l_2|}+1$\\
and  $ 2^{|l_1|+|l_2|}+\epsilon_22^{|l_1|}-\epsilon_1 2^{|l_2|}-1$, respectively. 
Whenever these positive integers are odd they are  the  Reidemeister number of the correspondent matrices,
which happens for $|l_1|,|l_2|>0$. Then in this case the Reidemeister number is the sum equals to 
$2^{|l_2|+1}(2^{|l_1|}-\epsilon_1)$. In order to find the complete spectrum we have to analyze the particular cases.
Suppose that $l_1=l_2=0$. Then in this case we have four possibilities for the pair $\epsilon_1,\epsilon_2$.
By straightforward calculation for each case either $det(Id-N)$ or $det(Id-MN)$ is zero (if not both) and we obtain Reidemeister
infinite. In details, for $(\epsilon_1,\epsilon_2)=(1,1)$ then $(det(Id-N), -det(Id-MN))=(0,0)$, 
for $(\epsilon_1,\epsilon_2)=(1,-1)$ then $(det(Id-N), -det(Id-MN))=(0,0)$, for $(\epsilon_1,\epsilon_2)=(-1,1)$ then $(det(Id-N), -det(Id-MN))=(0,-4)$ for $(\epsilon_1,\epsilon_2)=(-1,-1)$ then $(det(Id-N), -det(Id-MN))=(4,0)$  

Now  let $l_2=0$ and $l_1\ne 0$. By direct inspection for $\epsilon_2=1$ we obtain $det(N)=0$ and for 
$\epsilon_2=-1$ we obtain $det(MN)=0$, hence  we obtain Reidemeister infinite.

Finally  let $l_1=0$ and $l_2\ne 0$. If $\epsilon_1=1$ then $det(Id-N)=det(Id-MN)=0$ and Reidemeister is infinite. 
If $\epsilon_1=-1$ we get $det(Id-N)=-2(\epsilon_22^{|l_2|}-1)$ and  $det(Id-MN)=-2(\epsilon_22^{|l_2|}+1)$. After we compute $v_{P}$ of these numbers
and add them up we obtain the Reidemeister number $2^{|l_2|+1}$ for $|l_2|\geq 1$ or $2^k$ $k\geq 2$. But these numbers were obtained  already 
in  previous cases  and the result follows.

 Let us consider the case b).  Some of the calculations  are  similar and in this case we do not give all  details.  In this case the system of equations provide
$a=-d$. So the determinant of $N$ becomes $-a^2-bc=-a^2-b^2v^2=$ which is an invertible element. Since $uv=-1$ follows that  
$u=-\delta 2^{-t}$ and  $v=\delta 2^{t}$ for some integer $t$ and $\delta \in \{1,-1\}$.   

 The matrix $N$ is of the form 

$$
\begin{pmatrix}
a& b\\
bv^2& -a
\end{pmatrix}
$$
 \noindent and the matrix $MN$ is

$$
\begin{pmatrix}
-bv& a/v\\
av& bv
\end{pmatrix}.
$$

\noindent Our task is to commute $v_{P}$ of $det(Id-N)$ and $det(Id-MN)$. We have   $det(Id-N)=1-a^2-b^2v^2$ and 
 $det(Id-MN)=1-b^2v^2-a^2$. So we compute  $v_P$ of $det(Id-N)$.  

 Let $a=r2^m$ and $bv=s2^n$ where $r,s$ are odd numbers (possible negative) and $m,n$ integers. Because $a^2+(bv)^2$ is invertible
then we obtain that $r^2+s^2$ is necessarily 1 and we obtain as possible solutions  $a=\epsilon_12^m$ and $bv=0$ (or $b=0$ since $v\ne 0$)
 $a=0$ and $bv=\epsilon_2 2^n$. 

 For the case  $b=0$  we obtain $det(Id-N)= 1-a^2=1-2^{2m}$. For $m=0$ we obtain Reidemeister infinite otherwise we obtain
the total Reidemeister number $2(2^{2m}-1) \ for \  m>0$. For $a=0$ we  obtain $det(Id-N)= 1-b^2v^2=1-2^{2n}$. Then we obtain 
the same numbers as above and the result follows. 
 \end{proof}

The example studied  by Jabara in \cite{Jab}  is included in part a) above. Moreover, part a) above computes the Reidemeister spectrum 
of  the group.

For an arbitrary prime $p\ne 2$ we will have similar results.

\begin{proposition} Let  $\theta(1)$ be of the form:
$$
\begin{pmatrix}
r& 0 \\
0 & s
\end{pmatrix}
$$ Then we have the following cases:

a) If  $r=s=\pm 1$ then  
$Spec(\mathbb Z[1/p]^2\rtimes_{\theta} \mathbb Z)=\{2n |  n\in \mathbb N, \ (n,p)=1    \}\cup \{\infty\}$ where $(n,p)$   denote the gcd of $n$ and $p$.

b) If $r=-s=\pm 1$ then  the   Reidemeister spectrum of the group $\mathbb Z[1/p]^2\rtimes_{\theta}\mathbb Z$ is 
$Spec( \mathbb Z[1/p]^2\rtimes_{\theta}\mathbb Z)=\{ 2p^l(p^k\pm 1), 4p^l | l, k>0\} \cup\{  \infty\}$

c) If  $rs=1$ and $|r|\ne 1$      
then   $Spec(\mathbb Z[1/p]^2\rtimes_{\theta}\mathbb Z)=\{ 2(p^l \pm 1), 4 |l>0 \} \cup\{ \infty \}$

d) If either $r$ or $s$        does not have module equal to one, and $ rs \ne 1$ then 
 there is no automorphism of the group 
such that the induced on the quotient is multiplication by $-1$ and follows that  $Spec(\mathbb Z[1/p]^2\rtimes_{\theta}\mathbb Z)=\{\infty\}$.
\end{proposition}
\begin{proof} As in Proposition 3.5 we have the system  $a=ar^2$,   $b=brs$, $c=crs$ and $d=ds^2$. Part a) and d) follows as in 
Proposition 3.5.

For the part b) from the equations  $a=ar^2$,   $b=brs$, $c=crs$ and $d=ds^2$ follows  that the matrix $N$ is diagonal. Let
$a=\epsilon_1p^{l_1}$ and $d=\epsilon_2p^{l_2}$ since these elements are invertible. The $det(Id-N)=1-a-d+ad=\epsilon_1\epsilon_2p^{l_1+l_2}-\epsilon_1p^{l_1}-\epsilon_2p^{l_2}+1$ and 
$-det(Id-MN)=ad+a-d-1=\epsilon_1\epsilon_2p^{l_1+l_2}+\epsilon_1p^{l_1} -\epsilon_2p^{l_2}-1$.

Without loss of generality let us assume that $l_1,l_2 \geq 0$.
First let  $l_1=l_2=0$. Then  one of the two determinants is zero and the Reidemeister number is infinite. 
 Now   let $l_1=0$ and $l_2\ne 0$.  We have 
 $det(Id-N)=\epsilon_1\epsilon_2p^{l_2}-\epsilon_1-\epsilon_2p^{l_2}+1$ and 
$-det(Id-MN)=\epsilon_1\epsilon_2p^{l_2}+\epsilon_1 -\epsilon_2p^{l_2}-1$. If $\epsilon_1=1$ then $det=0$ 
in both cases and we 
have Reidemeister infinite. If $\epsilon_1=-1$ then we obtain 
 $det(Id-N)=-2\epsilon_2p^{l_2}+2$ and 
$-det(Id-MN)=-2\epsilon_2p^{l_2}-2$. Both numbers are not divisible by $p$ and the  Reidemeister number is 
the module of $4\epsilon_2 p^{l_2}$,  $l_2>0$.   Now   let $l_2=0$ and $l_1\ne 0$.  We have 
 $det(Id-N)=\epsilon_1\epsilon_2p^{l_1}-\epsilon_2-\epsilon_1p^{l_1}+1$ and 
$-det(Id-MN)=\epsilon_1\epsilon_2p^{l_1}+\epsilon_1p^{l_1} -\epsilon_2-1$. If $\epsilon_2=1$ then $det(Id-N)=0$ and we 
have Reidemeister infinite. If $\epsilon_2=-1$ then $det(Id-MN)=0$ and we 
have Reidemeister infinite.  Finally if $l_1, l_2>0$ then the two numbers 
$\epsilon_1\epsilon_2p^{l_1+l_2}-\epsilon_1p^{l_1}-\epsilon_2p^{l_2}+1$ and 
$\epsilon_1\epsilon_2p^{l_1+l_2}+\epsilon_1p^{l_1} -\epsilon_2p^{l_2}-1$ 
are not divisible by $p$ and the Reidemeister number is the module of 
$2\epsilon_1\epsilon_2p^{l_1+l_2} -2\epsilon_2p^{l_2}=2\epsilon_2p^{l_2}(\epsilon_1p^{l_1}-1)$. 
So the Reidemeister numbers are of the form  $2p^l(p^k\pm 1)$, $l,k>0$,  and the result follows.

For the part c)  from the equations  $a=ar^2$,   $b=brs$, $c=crs$ and $d=ds^2$ follows $a=d=0$. So the matrix $N$ 
 is of the form 

$$
\begin{pmatrix}
0& b\\
c& 0
\end{pmatrix}.
$$

\noindent  and  $MN$ is 

$$
\begin{pmatrix}
0& rb\\
sc & 0
\end{pmatrix}.
$$

\noindent By direct calculation, using that $rs=1$,  follows that $det(Id-N)=det(Id-MN)=1-bc$.
Since $bc$ is invertible then we can write $b=\epsilon_1p^{l_1}$ and  $c=\epsilon_2p^{l_2}$
and follows that $det(Id-N)=1-\epsilon_1\epsilon_2p^{l_1+l_2}$. If $l_1+l_2=0$ then we obtain 
for the determinant 0 or 2. So we obtain for Reidemeister number infinite and 4. If $l_1+l_2\ne 0$ then we obtain as Reidemeister numbers
the numbers of the form $2(p^l \pm 1),   l>0$ and the result follows.

\end{proof}

\begin{proposition} Let  $\theta(1)$ be of the form:
$$
\begin{pmatrix}
0 & u \\
v & 0
\end{pmatrix}
$$ Then we have the following cases: 

a) If $ uv=1$     then    the  Reidemeister spectrum of the group $\mathbb Z[1/p]^2\rtimes_{\theta}\mathbb Z$ is 
$Spec(\mathbb Z[1/p]^2\rtimes_{\theta}\mathbb Z)=\{ 2 p^l(p^l\pm 1) | l>0\} \cup\{\infty \}$.\\

b) If $ uv=-1$ and then   the  Reidemeister spectrum of the group $\mathbb Z[1/p]^2\rtimes_{\theta}\mathbb Z$ is 
$Spec(\mathbb Z[1/p]^2\rtimes_{\theta}\mathbb Z)=\{2 (p^l\pm 1) | l>0 \} \cup\{\infty \}$.\\

c) If   $u^2v^2\ne  1$ then    there is no automorphism of the group such that the induced on the quotient is multiplication by $-1$ and follows that  $Spec(\mathbb Z[1/p]^2\rtimes_{\theta}\mathbb Z)=\{\infty\}$.
 \end{proposition}
\begin{proof} The proof follows the same steps as the proof of Proposition 3.6 and it is simpler. The matrix $N$ 
$$
\begin{pmatrix}
a& b\\
c& d
\end{pmatrix}
$$

\noindent must satisfy the equations $uvd=a, u^2c=b, bv^2=c, auv=d$. It follows that $a=(uv)^2a$,  $d=(uv)^2d$, $b=(uv)^2b$, $c=(uv)^2c$.
So part c) follows promptly from these equations.  

 Let us consider the case a).  In this case the system of equations provide
$a=d$. So the determinant of $N$ becomes $a^2-bc=a^2-b^2v^2=(a+bv)(a-bv)$ which is an invertible element, so $(a+bv)$ and $(a-bv)$ are also 
invertible. Since $uv=1$ follows that  $u=\delta p^{-t}$ and  $v=\delta p^{t}$ for some integer $t$ and $\delta \in \{1,-1\}$.   

 The matrix $N$ is of the form 

$$
\begin{pmatrix}
a& b\\
bv^2& a
\end{pmatrix}
$$
 \noindent and the matrix $MN$ is

$$
\begin{pmatrix}
bv& a/v\\
av& bv
\end{pmatrix}.
$$

\noindent Our task is to compute $v_{P}$ of $det(Id-N)$ and $det(Id-MN)$. We have   $det(Id-N)=1-2a+a^2-b^2v^2$ and 
 $det(Id-MN)=1-2bv+b^2v^2-a^2$ so  $-det(Id-MN)=-1+2bv+a^2-b^2v^2$. We compute $v_{P}$ of  $-det(Id-MN)$ which is 
the same as $v_P$ of  $det(Id-MN)$.

 Since $a+bv$ and $a-bv$ are  invertible they are of the form $a+bv=\epsilon_1p^{l_1}$ and   $a-bv=\epsilon_2p^{l_2}$.
Therefore $2a=\epsilon_1 p^{l_1} +\epsilon_2 p^{l_2}$ and $2bv=\epsilon_1p^{l_1}-\epsilon_2p^{l_2}$.
In order to have $\epsilon_1 p^{l_1} +\epsilon_2 p^{l_2}$ divisible by 2 we need to have $l_1=l_2=l$ and 
either $\epsilon_1=\epsilon_2$ or  $\epsilon_1=-\epsilon_2$. In the first case we have $a=\epsilon_1 p^l$ and $2bv=0$ so $b=0$.
In the latter case we have $a=0$ and $bv=\epsilon_1s^l$. Now we compute the Reidemeister for each of these two cases. Let 
 $a=\epsilon_1 p^l$ and  $b=0$. Then $det(Id-N)=1-2a+a^2$ and  $-det(Id-MN)=a^2-1$. If $l=0$ then $det(Id-MN)=0$ and 
we get Reidemeister infinite. If $l\ne 0$ then the determinants are not divisible by $p$ and we  get as Reidemister number
the module of 
$2a^2-2a=2a(a-2)=2\epsilon p^l(\epsilon p^l-1)$. So the Reidemeister numbers are of the form $2 p^l(p^l\pm 1)$, $l\ne 0$.
So the result follows.

 Let us consider the case b).  Some of the calculations are  similar to the corresponding  case ofthe  Proposition 3.6   and in this case 
we do not give the details.  The system of equations provide
$a=-d$. So the determinant of $N$ becomes $-a^2-bc=-a^2-b^2v^2$ which is an invertible element. Since $uv=-1$ follows that  
$u=-\delta 2^{-t}$ and  $v=\delta 2^{t}$ for some integer $t$ and $\delta \in \{1,-1\}$.   

 The matrix $N$ is of the form 

$$
\begin{pmatrix}
a& b\\
bv^2& -a
\end{pmatrix}
$$
 \noindent and the matrix $MN$ is

$$
\begin{pmatrix}
-bv& a/v\\
av& bv
\end{pmatrix}.
$$

\noindent Our task is to commute $v_{P}$ of $det(Id-N)$ and $det(Id-MN)$. We have   $det(Id-N)=1-a^2-b^2v^2$ and 
 $det(Id-MN)=1-b^2v^2-a^2$. So we compute  $v_P$ of $det(Id-N)$.  

 Let $a=rp^m$ and $bv=sp^n$ where $r,s$ are relatively prime with $p$. Because $a^2+(bv)^2$ is invertible
then we obtain that $r^2+s^2$ is necessarily 1 and we obtain as  possible solution  $a=\epsilon_1p^m$ 
and $bv=0$ (or $b=0$ since $v\ne 0$),  or 
 $a=0$ and $bv=\epsilon_2 p^n$. 

 For the case  $b=0$  we obtain $det(Id-N)= 1-a^2=1-p^{2m}$. For $m=0$ we obtain Reidemeister infinite otherwise we obtain
the total Reidemeister number $2(p^{2m}-1), m>0$. For $a=0$ we  obtain $det(Id-N)= 1-b^2v^2=1-p^{2n}$. Then we obtain the same numbers 
as above and the result follows. 

 \end{proof}

\section{The semi-direct products  $\mathbb Q^n\rtimes Z$, $\mathbb Z[1/p]^n\rtimes Z$,  $n> 2$}

There are some  of the above results that extend easily  to the groups $\mathbb Q^n\rtimes Z$, $\mathbb Z[1/p]^n\rtimes \mathbb Z$,  $n> 2$.
One of the results refer to the group  $\mathbb Q^n\rtimes \mathbb Z$. The abelian group $\mathbb Q^n$  has the property that it does not have a 
subgroup of finite index.
Then   an immediate  consequence of this fact    is that the   Reidemeister number of any homomorphism is either 1
or infinite. So  the following  result holds.

\begin{proposition}  The    Reidemeister spectrum of $\mathbb Q^n\rtimes_{\theta}\mathbb Z$ is either $\{\infty\}$ or $\{2\}\cup\{\infty\}$.
The $Spec(\mathbb Q^n\rtimes_{\theta}\mathbb Z)=\{\infty\}$ if  there is no an invertible matrix $N$ over $\mathbb Q$
such that $N=MNM$ and $det(Id-N)$ and $det(Id-MN)$ are non zero. Here $\theta(1)=M$. Otherwise the 
Reidemeister spectrum is    $\{2\}\cup\{\infty\}$.
\end{proposition} 
\begin{proof}(sketch) The only possible automorphism which can have Reidemeister finite is one such that the induced homomorphism 
on  the quotient $\mathbb Z$ is multipliction by $-1$. From the procedure described in the begin of section 3 we have to compute 
the Reidemeister of the homomorphism given by $N$ and $MN$. But a homomorphism of $\mathbb Q^n$ is either surjective or has cokernel infinite. 
So the sum of the Reidemeister of the two homomorphisms is either infinite or 2. The case where both homomorphisms have Reidemeister
1 corresponds to say that $det(Id-N)$ and $det(Id-MN)$   are non zero and   the result follows. 
\end{proof}

It is easy  to construct examples which illustrated both cases, i.e. when  the    Reidemeister spectrum of $\mathbb Q^n\rtimes_{\theta}\mathbb Z$ is  $\{\infty\}$ 
and when it is  $\{2\}\cup\{\infty\}$.

Now let $p$ be an arbitrary prime. It is not difficult to construct action $\theta(1)$ which has diagonal matrix such that
there is no automorphism of the group  such that the induced automorphims on the quotient $\mathbb Z$ is  $-id$. This gives the  examples of groups with the $R_{\infty}$ property. 

\subsection{ Final comments}
 Further  calculation of the Reidemeister spectrum  for some of the groups of the form $A^n\rtimes_{\theta} Z$, with $A$ either $\mathbb Q$ or $\mathbb Z[1/p]$,
not studied above,  is in progress.

 It is interesting  to investigate
 $Spec_NC(X)= \{ N(f)| f \in Map(X) \}$ 
and   $Spec_NH(X)= \{ N(f)| f \in Homeo(X) \}$
  i.e., the {\it  Nielsen spectrum for continious maps of  the space} $X$
and  {\it  Nielsen spectrum for homeomorphisms  of  the space} $X$.
\begin{example}
 For selfmaps $f$ of nilmanifold, either $N(f)=0$ or $N(f)=R(f)$ \cite{fhw}.
 It is easy to see that  $Spec_NH(S^1) = \{2\} \cup \{0 \}$, 
and,  for $n \geq 2$,  the spectrum is of $T^n$ is  full, i.e.
$ Spec_NH(T^n) = {\mathbb N}\cup \{0 \}$.
Let $ Nil_{rl}$ be the nilmanifold with the fundamental group being  a  free nilpotent group of rank $r$ and class
$l$. It follows from Introduction that  for $Nil_{22}$ we have   $Spec_NH(Nil_{22}) = 2{\mathbb N} \cup \{0 \}$; for $Nil_{23}$ we have  $ Spec_NH(Nil_{23}) = \{ 2k^2 | k \in {\mathbb N}\} \cup
\{0 \}$  and for $Nil_{32}$ we have 
$ Spec_NH (Nil_{32}) = \{2n - 1| n \in {\mathbf N}\} \cup \{
4n | n \in {\mathbf N}\} \cup \{0 \}$ .

\end{example}

\bibliography{ref}
\bibliographystyle{amsplain}

\end{document}